\documentclass[preprint,11pt]{elsarticle}
\usepackage{lineno,hyperref}
\modulolinenumbers[5]

\usepackage{color}
\usepackage{amsmath}
\usepackage{amsfonts,amsthm,amssymb}
\usepackage{amsfonts}
\usepackage{graphics}
\usepackage{pstricks}
\usepackage{graphicx}
\usepackage{cleveref}
\usepackage{extarrows,chngpage,array,float}
\usepackage{amssymb}
\usepackage{latexsym,bm}
\newtheorem{theorem}{Theorem}[section]
\newtheorem{lemma}[theorem]{Lemma}

\newtheorem{conjecture}[theorem]{Conjecture}

\allowdisplaybreaks   

\numberwithin{equation}{section}

\makeatletter
\def\ps@pprintTitle{}
\makeatother

\begin{document}

\makeatletter
\def\ps@pprintTitle{}
\makeatother

\begin{frontmatter}

\title{Proof of a Brown-Mol conjecture on subtree roots}

\author{Xian'an Jin$^{a,b}$, \ \ Tianlong Ma$^{c,}$\footnote{Corresponding author.}, \ \ Yi Wang$^{d}$
	\\[1ex]
\small $^a$School of Mathematics and Statistics\\[-0.8ex]
\small Qinghai Minzu University\\[-0.8ex]
\small P. R. China\\
\small $^b$School of Mathematical Sciences\\[-0.8ex]
\small Xiamen University\\[-0.8ex]
\small P. R. China\\
\small $^c$School of Science\\[-0.8ex]
\small Jimei University\\[-0.8ex]
\small P. R. China\\
\small $^d$School of Mathematical Sciences\\[-0.8ex]
\small  Anhui University\\[-0.8ex]
\small P. R. China\\
\small\tt Email: xajin@xmu.edu.cn, tianlongma@aliyun.com, wangy@ahu.edu.cn}

\begin{abstract}
The subtree polynomial of a tree is the generating function that
enumerates its subtrees according to their orders. Brown and Mol
conjectured that every subtree root of a tree of order $n\ge 2$
lies in the disk
\[
\left\{z\in\mathbb C:
|z|\le 1+\sqrt[n-1]{n-1}
\right\}.
\]
We prove this conjecture by introducing a recursive comparison method
based on an extremal problem over integer compositions. We further
characterize the equality case: the upper bound is attained if and
only if $n$ is even and the tree is the star; in this
case the unique boundary root is $-1-\sqrt[n-1]{n-1}$.
We also show that every nonzero subtree root $z$ satisfies
\[
|z|>\sqrt[n-1]{n-1}-1.
\]
The lower bound is asymptotically sharp as $n\to\infty$. For odd
$n$, although the upper bound is not attained, it is asymptotically
sharp as $n\to\infty$.

\end{abstract}


\end{frontmatter}

\section{Introduction}

Let $T$ be a tree of order $n$, and let $s_k(T)$ denote the number of
subtrees of $T$ with $k$ vertices.
The \emph{subtree polynomial} of $T$, introduced by Jamison in \cite{1983On}, 
is
\[
\Phi_T(z)=\sum_{k=1}^{n}s_k(T)z^k.
\]

Subtree polynomials and related subtree statistics have been studied
extensively. Jamison~\cite{1983On,1984Monotonicity} initiated the study
of the mean order of subtrees and established fundamental extremal and
monotonicity results. Vince and Wang~\cite{Vince} obtained sharp bounds
on the average order of subtrees in series-reduced trees, and
Haslegrave~\cite{Haslegrave} further investigated the corresponding
extremal average subtree densities. More recently, Cambie, Wagner, and
Wang~\cite{CambieWagnerWang} studied the maximum mean subtree order.
Luo, Xu, Wagner, and Wang~\cite{LuoXuWagnerWang} proved Jamison's
edge-contraction conjecture for pendant edges, and Wang~\cite{WangContraction}
subsequently settled the conjecture in full. For further results on
subtree enumeration, local and global mean subtree orders, subtree-order
distributions, and related extremal problems, we refer the reader to
\cite{CambieChenHaoTokar,CameronMol,ChenWeiLian,Ralaivaosaona,Sills,VinceConnected,Wagner,Yan}.

The zeros of $\Phi_T(z)$  are
called \emph{subtree roots}.
In \cite{Brown}, Brown and Mol established a universal upper bound on the moduli
of subtree roots.

\begin{theorem}\emph{\cite{Brown}}\label{K1,3}
	The subtree roots of every tree $T$ lie in the disk
	\[\{z\in \mathbb{C}: |z|\leq 1 + \sqrt[3]{3}\},\]
	and the only tree with a subtree root on the boundary of this disk is $K_{1,3}$.
\end{theorem}

Brown and Mol further proposed an order-dependent refinement of Theorem \ref{K1,3}.

\begin{conjecture}\emph{\cite{Brown}}\label{C1}
	Let $T$ be a tree of order $n\geq 2$. Then the subtree roots of $T$ lie in the disk
	\[\{z\in \mathbb{C}: |z|\leq 1 + \sqrt[n-1]{n-1}\}.\]
\end{conjecture}

For $n\ge4$, the radius in Conjecture~1.2 is strictly decreasing
with $n$, so the
inductive argument used for the universal disk cannot be applied directly:
passing to a proper subtree changes the target radius in the wrong direction.
To overcome this difficulty, we introduce a recursively defined quantity that
controls, at a fixed modulus, the ratio between the subtree polynomial of a
vertex-deleted forest and the local subtree polynomial at the deleted vertex.
A sharp estimate for this ratio at the conjectured radius then yields the
desired zero-free region. 

Our main result resolves Conjecture \ref{C1} and characterizes the equality case.

\begin{theorem}\label{thm:main}
	Let $T$ be a tree of order $n\geq 2$. Then the subtree roots of $T$ lie in the disk
	\[\{z\in \mathbb{C}: |z|\leq 1 + \sqrt[n-1]{n-1}\}.\] 
    Moreover, $T$ has a subtree root on the boundary $|z|=1+(n-1)^{1/(n-1)}$ 
if and only if $n$ is even and  $T\cong K_{1,n-1}$, and  
in this case the unique boundary root is $z=-1-(n-1)^{1/(n-1)}$.

\end{theorem}

The remainder of the paper is organized as follows. In Section~2, we develop the
local decomposition formulas and the recursive comparison estimate needed for
the proof of Theorem~\ref{thm:main}. 
In Section~3, we  prove Theorem~\ref{thm:main}. In Section~4, we  provide a lower bound on the modulus of nonzero subtree roots and discuss the sharpness of our upper and lower bounds. In the final section, we recall a stronger conjecture of Brown and
Mol that implies Conjecture \ref{C1}.

\section{Preliminaries}

We first list several related definitions and observations on the subtree polynomial which have appeared in the literature. In \cite{Brown}, the definition of subtree polynomial was extended to forests, as forests were regarded as vertex-deleted subgraphs of trees. If $F$ is a forest with components $T_1,\ldots,T_k$, then the subtree polynomial of $F$ is given by
\[\varPhi_{F} (z)=\sum_{i=1}^{k}\varPhi_{T_i}(z).\]
Let $T$ be a tree of order $n$, and let $v\in V(T)$.
For each $k\in\{1,\ldots,n\}$, let $s_k(T,v)$ denote the number of subtrees of $T$ of order $k$ containing $v$. Then the \emph{local subtree polynomial} of $T$ at $v$, introduced in \cite{1983On}, is  defined by
\[\varPhi_{T,v} (z)=\sum_{k=1}^{n}s_k(T,v)z^k.\]
For any vertex $u$ of a tree $T$, every subtree of $T$ either contains $u$ or is contained in $T-u$.
Consequently,
\begin{equation}
	\label{eq:decomposition}
	\Phi_T(z)=\Phi_{T,u}(z)+\Phi_{T-u}(z).
\end{equation}
Let $T_1,\dots,T_d$ be the components of $T-u$, where $d=d_T(u)$, the degree of the vertex $u$ in $T$, and let $u_i$ be the
unique neighbor of $u$ in $T_i$. A subtree containing $u$ is obtained by
choosing, independently for each $i$, either no vertex of $T_i$ or a
subtree of $T_i$ containing $u_i$. Hence
\begin{equation}
	\label{eq:rooted-product}
	\Phi_{T,u}(z)
	=
	z\prod_{i=1}^{d}\bigl(1+\Phi_{T_i,u_i}(z)\bigr).
\end{equation}

We shall use the following lower bound given by Brown and Mol in \cite{Brown}.

\begin{lemma}\emph{\cite{Brown}}\label{lowerbound}
	Let $T$ be a tree of order $n$, and let $z\in \mathbb{C}$ with $|z|\geq 2$. Then for every vertex $v$ of $T$,
	$$|\Phi_{T,v}(z)|\geq |z|(|z|-1)^{n-1}.$$
\end{lemma}

In order to estimate $|\Phi_{T-u}(z)|/|\Phi_{T,u}(z)|$ in the next section, we introduce several auxiliary functions and establish the required
monotonicity properties.

Let $c_1,\dots,c_d>0$. For $x_1, \ldots, x_d >1$,  define
\[
H_{\boldsymbol{c}}(x_1,\dots,x_d)
=
\frac{\sum_{i=1}^{d}c_ix_i}
{\prod_{i=1}^{d}(x_i-1)}.
\]

\begin{lemma}
	\label{lem:H-coordinatewise}
	If $1<x_i\le y_i$ for every $i=1,\dots,d$, then
	\[
	H_{\boldsymbol{c}}(y_1,\dots,y_d)
	\le
	H_{\boldsymbol{c}}(x_1,\dots,x_d).
	\]
	Moreover, the inequality is strict if there exists an index $j$ such that $x_j<y_j$.
\end{lemma}

\begin{proof}
	For each $j$, since \[
	\frac{\partial H_{\boldsymbol{c}}}{\partial x_j}
	=
	-\frac{c_j+\sum_{i\ne j}c_i x_i}
	{(x_j-1)\prod_{i=1}^d (x_i-1)}
	<0,
	\] 
	it follows that $H_{\boldsymbol{c}}(x_1,\dots,x_d)$ 
	is strictly decreasing in each variable $x_j$ on $(1,+\infty)$.
	
	Set
	\[
	\boldsymbol{w}_{0}=(x_1,\dots,x_d),
	\]
	and, for $j=1,\dots,d$, define
	\[
	\boldsymbol{w}_{j}
	=
	(y_1,\dots,y_j,x_{j+1},\dots,x_d).
	\]
	Clearly, 
	$$\boldsymbol{w}_{d}
	=
	(y_1,y_2,\dots,y_d). $$ For each $j$, the vectors $\boldsymbol{w}_{j-1}$ and
	$\boldsymbol{w}_{j}$ differ only in their $j$-th coordinate, and $x_j\leq y_j$.	
	Since $H_{\boldsymbol{c}}(x_1,\dots,x_d)$ 
	is strictly decreasing in each variable $x_j$ on $(1,\infty)$, we have
	\[
	H_{\boldsymbol{c}}\bigl(\boldsymbol{w}_{j}\bigr)
	\le
	H_{\boldsymbol{c}}\bigl(\boldsymbol{w}_{j-1}\bigr)
	\]
	with strict inequality whenever $x_j<y_j$. Thus 
	\[H_{\boldsymbol{c}}\bigl(\boldsymbol{w}_{d}\bigr)\le H_{\boldsymbol{c}}\bigl(\boldsymbol{w}_{d-1}\bigr)\le\ldots \le H_{\boldsymbol{c}}\bigl(\boldsymbol{w}_{0}\bigr)\]
	This completes the proof.
\end{proof}

For $n\ge 1$ and $r\ge 2$, define
\[
\lambda_n(r)=r(r-1)^{n-1}.
\]
For a positive integer $n$, let $\mathcal{C}(n)$ denote the set of all
compositions of $n$ into positive integers. Define $B_1(r)=0$, and
for $n\ge 2$ define
\begin{equation*}
	B_n(r)
	=
	\max_{(n_1,\dots,n_d)\in\mathcal{C}(n-1)}
	\frac{
		\displaystyle
		\sum_{i=1}^{d}
		\lambda_{n_i}(r)\bigl(1+B_{n_i}(r)\bigr)
	}{
		\displaystyle
		r\prod_{i=1}^{d}\bigl(\lambda_{n_i}(r)-1\bigr)
	}.
\end{equation*}
This recursion is well defined because every part $n_i$ is
smaller than $n$, the set $\mathcal C(n-1)$ is finite, and
$\lambda_{n_i}(r)-1\ge 1>0$
for $r\ge 2$.
For illustration, the first two nontrivial functions in this recursion can be computed explicitly:
\[
B_2(r)
=
\frac{1}{r-1},
\]
and
\[
B_3(r)
=
\max\left\{
\frac{r}{r^2-r-1},
\frac{2}{(r-1)^2}
\right\}=\begin{cases}
	\displaystyle \frac{2}{(r-1)^2},
	& 2\le r\le 1+\sqrt{2},\\[2ex]
	\displaystyle \frac{r}{r^2-r-1},
	& r\ge 1+\sqrt{2}.
\end{cases}
\]

\begin{lemma}
	\label{lemma:Bdecreasing}
	For every integer $n\ge 2$, the function $B_n(r)$ is strictly decreasing in 
	$r$ on $[2,+\infty)$.
\end{lemma}

\begin{proof}
	We use induction on $n$.
	For $n=2$,
	\[
	B_2(r)=
	\frac{1}{r-1},
	\]
	which is strictly decreasing for $r\ge 2$.
	
	Now let $n\ge 3$, and assume that $B_j(r)$ is strictly decreasing for every
	$2\le j<n$. Recall that $B_1(r)=0$, so $B_1(r)$ is nonincreasing as well.
	
	Fix a composition
	$\rho=(n_1,\dots,n_d)\in\mathcal{C}(n-1)$, 
	and define
	\begin{equation*}
		E_{\rho}(r)
		=
		\frac{
			\displaystyle
			\sum_{i=1}^{d}
			\lambda_{n_i}(r)\left(1+B_{n_i}(r)\right)
		}{
			\displaystyle
			r\prod_{i=1}^{d}
			\left(\lambda_{n_i}(r)-1\right)
		}.
	\end{equation*}
	We first prove that $E_{\rho}(r)$ is strictly decreasing.
	
	Take two real numbers $r_1$ and $r_2$ such that $2\le r_1<r_2$. For $j\in\{1,2\}$, write
	\[
	x_i^{(j)}=\lambda_{n_i}(r_j),
	\]
	and
	\[
	c_i^{(j)}=1+B_{n_i}(r_j).
	\]
	Since, for every integer $k\ge 1$, the function $\lambda_k(r)=r(r-1)^{k-1}$
	is strictly increasing in $r\in [2,\infty)$, we have $x_i^{(1)}<x_i^{(2)}$
	for every $i$. By the induction hypothesis and the fact that $B_1(r)=0$ when $n_i=1$,
	we have $c_i^{(2)}\le c_i^{(1)}$
	for every $i$.
	
	For a coefficient vector $\boldsymbol{c}=(c_1,\dots,c_d)$, write
	\[
	H_{\boldsymbol{c}}(x_1,\dots,x_d)
	=
	\frac{\sum_{i=1}^d c_i x_i}{\prod_{i=1}^d(x_i-1)}.
	\]
	Then
	\[
	E_{\rho}(r_j)
	=
	\frac{1}{r_j}
	H_{\boldsymbol{c}^{(j)}}\!
	\left(x_1^{(j)},\dots,x_d^{(j)}\right).
	\]
	Note that $x_i^{(j)}>1$ for $1\leq i\leq d$ and $j=1,2$. Since $$\frac{\partial H_c}{\partial c_j} = \frac{x_j}{\prod_{i=1}^{d} (x_i - 1)} > 0,$$ the function $H_{\boldsymbol{c}}(x_1,\dots,x_d)$ is strictly increasing in each coefficient
	$c_i$. Therefore
	\begin{align*}
		H_{\boldsymbol{c}^{(2)}}\!
		\left(x_1^{(2)},\dots,x_d^{(2)}\right)
		\le
		H_{\boldsymbol{c}^{(1)}}\!
		\left(x_1^{(2)},\dots,x_d^{(2)}\right).
	\end{align*}
	Since $x_i^{(1)}<x_i^{(2)}$ for every $i$,
	by Lemma~\ref{lem:H-coordinatewise}, we have
	\begin{align*}
		H_{\boldsymbol{c}^{(1)}}\!
		\left(x_1^{(2)},\dots,x_d^{(2)}\right)
		<
		H_{\boldsymbol{c}^{(1)}}\!
		\left(x_1^{(1)},\dots,x_d^{(1)}\right).
	\end{align*}
	Then
	\begin{equation*}
		\label{eq:H-strict}
		H_{\boldsymbol{c}^{(2)}}\!
		\left(x_1^{(2)},\dots,x_d^{(2)}\right)
		<
		H_{\boldsymbol{c}^{(1)}}\!
		\left(x_1^{(1)},\dots,x_d^{(1)}\right).
	\end{equation*}
	
	Since $1/r_2<1/r_1$, $H_{\boldsymbol{c}^{(2)}}\!
	\left(x_1^{(2)},\dots,x_d^{(2)}\right)\geq 0$ and $H_{\boldsymbol{c}^{(1)}}\!
	\left(x_1^{(1)},\dots,x_d^{(1)}\right)\geq 0$, we have
	\begin{align*}
		E_{\rho}(r_2)
		=
		\frac{1}{r_2}
		H_{\boldsymbol{c}^{(2)}}\!
		\left(x_1^{(2)},\dots,x_d^{(2)}\right)
		<
		\frac{1}{r_1}
		H_{\boldsymbol{c}^{(1)}}\!
		\left(x_1^{(1)},\dots,x_d^{(1)}\right)
		=
		E_{\rho}(r_1).
	\end{align*}
	Thus, for every fixed composition
	$\rho\in\mathcal C(n-1)$, the function $E_{\rho}$ is strictly decreasing.
	
	Finally, choose a composition $\rho_2\in \mathcal{C}(n-1)$ for which
	$B_n(r_2)=E_{\rho_2}(r_2)$. Then
	\[
	B_n(r_1)
	\ge E_{\rho_2}(r_1)
	>
	E_{\rho_2}(r_2)
	=
	B_n(r_2).
	\]
	Hence $B_n(r)$ is strictly decreasing in 
	$r$ on $[2,+\infty)$.
\end{proof}

For $n\ge 5$, define
\[
a_n=(n-1)^{1/(n-1)}
\]
and
\[
\gamma_n=1+a_n.
\]

We now establish two elementary inequalities involving $a_n$.
\begin{lemma}
	\label{lem:ak}
	For every integer $k$ with $2\le k\le n-1$, one has $a_n^k\le k$,
 with equality if and only if either $k=n-1$, or $n=5$ and $k=2$.
\end{lemma}

\begin{proof}
	The function $x^{1/x}$ is decreasing for $x \geq e$. If
	$3\le k\le n-1$, then
	\[
	a_n=(n-1)^{1/(n-1)}\le k^{1/k},
	\]
	and hence $a_n^k\le k$. For $k=2$, since $n-1\ge 4$,
	\[
	a_n=(n-1)^{1/(n-1)}\le 4^{1/4}=\sqrt{2}.
	\]
	Thus $a_n^2\le 2$.
    
Since $x^{1/x}$ is strictly decreasing on $[e,\infty)$,
equality for $3\le k\le n-1$ holds if and only if $k=n-1$.
For $k=2$, equality holds if and only if $n-1=4$, that is,
$n=5$.
\end{proof}

For $1\le k\le n-1$, define
\[
q_k=1+a_n^{-1}-a_n^{-k}.
\]
Then $q_k\ge 1$, and
\begin{equation}
	\label{eq:qfactor}
	(1+a)a_n^{k-1}-1=a_n^kq_k.
\end{equation}

\begin{lemma}
	\label{lem:partition}
	Let $m_1,\dots,m_d$ be positive integers such that
	$\sum_{i=1}^d m_i=m\le n-1$. Then
	\[
	\sum_{i=1}^{d}\bigl(a_n^{m_i-1}+m_i-1\bigr)
	\le
	m\prod_{i=1}^{d}q_{m_i},
	\]
with equality if and only if one of the following conditions holds:
\begin{enumerate}
    \item [(1)]$m_1=\cdots=m_d=1$;
    \item [(2)]$d=1$ and $m_1=m=n-1$;
    \item [(3)]$n=5$, $d=1$, and $m_1=m=2$.
\end{enumerate}
\end{lemma}

\begin{proof}
	We first show that, for $1\le k\le n-1$,
	\[
	a_n^{k-1}+k-1\le kq_k.
	\]
	For $k=1$, equality holds. For $k\ge 2$, Lemma~\ref{lem:ak} gives
	\[
	a_n^k\bigl(kq_k-a_n^{k-1}-k+1\bigr)
	=
	(k-a_n^k)(a_n^{k-1}-1)
	\ge 0.
	\]
Then 
\[
a_n^{k-1}+k-1\le kq_k.
\]
Therefore,
\begin{equation}
\label{eq:partition-first}
\sum_{i=1}^{d}
\bigl(a_n^{m_i-1}+m_i-1\bigr)
\le
\sum_{i=1}^{d}m_iq_{m_i}.
\end{equation}
Since every $q_{m_i}\ge 1$, we have
\begin{equation}
\label{eq:partition-second}
\sum_{i=1}^{d}m_iq_{m_i}
\le
\sum_{i=1}^{d}m_i\prod_{j=1}^{d}q_{m_j}
=
m\prod_{j=1}^{d}q_{m_j}.
\end{equation}
This proves the inequality.

We now characterize the equality cases. Equality in 
\[
\sum_{i=1}^{d}\bigl(a_n^{m_i-1}+m_i-1\bigr)
	\le
	m\prod_{i=1}^{d}q_{m_i}
\]  holds if and only if equality holds in both 
Inequalities (\ref{eq:partition-first}) and (\ref{eq:partition-second}). 
For a fixed $k$, since 
\[
a_n^k\bigl(kq_k-a_n^{k-1}-k+1\bigr)=(k-a_n^k)(a_n^{k-1}-1),
\]
equality in $a_n^{k-1}+k-1\le kq_k$
holds if and only if $(k-a_n^k)(a_n^{k-1}-1)=0$. By Lemma~\ref{lem:ak}, 
$k-a_n^k=0$ if and only if $k=n-1$ or $(n,k)=(5,2)$, and clearly, $a_n^{k-1}-1=0$ if and only if $k=1$. Therefore
equality in Inequality (\ref{eq:partition-first}) holds if and only if every $m_i$ satisfies one of these three
conditions: $m_i=n-1$, $(n, m_i)=(5,2)$ or $m_i=1$.

If $m_i=n-1$ for some $i$, then we have $d=1$
as $\sum_{i=1}^d m_i\le n-1$, and hence
$m_1=m=n-1$. In this case equality in Inequality \ref{eq:partition-second} is automatic.  

If $m_i=1$ for every $i$, then $q_{m_i}=1$ for every $i$,
so equality also holds in Inequality \ref{eq:partition-second}.

The only remaining case is $n=5$, with at least one part equal to $2$ and all the remaining parts belonging to $\{1,2\}$.
If $d\ge2$, then $q_2>1$, and hence Inequality (\ref{eq:partition-second}) is
strict. Therefore equality is possible only when $d=1$, which
implies $m_1=m=2$.

Thus we obtain the desired condition for equality.
\end{proof}

We now combine the preceding two elementary inequalities to obtain the required upper bound for $B_m(\gamma_n)$. 
\begin{lemma}
	\label{lem:criticalB}
	Let $n$ be an integer with $n\geq 5$. Then, for each integer $m$ with $1\le m\le n$,
	\[
	B_m(\gamma_n)= \frac{m-1}{a_n^{m-1}}.
	\]
    Moreover, when $m=n$, the maximum in the definition of
$B_n(\gamma_n)$ is obtained precisely by the two compositions $(1,1,\dots,1)\in\mathcal{C}(n-1)$ and $(n-1)\in\mathcal{C}(n-1)$.
\end{lemma}

\begin{proof}
	For a fixed $n\geq 5$,	we use induction on $m$. The assertion is immediate for $m=1$. 
    Now let $2\le m\le n$, and assume that for every $1\le l<m$,
\[
B_l(\gamma_n)=\frac{l-1}{a_n^{l-1}}.
\]
Let $(m_1,\dots,m_d)\in\mathcal{C}(m-1)$. Since $\lambda_k(\gamma_n)=(1+a_n)a_n^{k-1}$, the induction hypothesis gives
\[
\lambda_{m_i}(\gamma_n)\bigl(1+B_{m_i}(\gamma_n)\bigr)=(1+a_n)\bigl(a_n^{m_i-1}+m_i-1\bigr).
\]
Therefore 
\[
\sum_{i=1}^{d}\lambda_{m_i}(\gamma_n)\bigl(1+B_{m_i}(\gamma_n)\bigr)=(1+a_n)\sum_{i=1}^{d}\bigl(a_n^{m_i-1}+m_i-1\bigr).
\]
By Equation~\eqref{eq:qfactor} and the fact $m_1+\cdots+m_d=m-1$, 
\[\prod_{i=1}^{d}\bigl(\lambda_{m_i}(\gamma_n)-1\bigr)=\prod_{i=1}^{d}\bigl((1+a_n)a_n^{m_i-1}-1\bigr)=a_n^{m-1}\prod_{i=1}^{d}q_{m_i}.\]
Therefore, for any $(m_1,\dots,m_d)\in\mathcal{C}(m-1)$, we have
\begin{equation}\label{eq:maxdecom}
\frac{\sum_{i=1}^{d}\lambda_{m_i}(\gamma_n)\bigl(1+B_{m_i}(\gamma_n)\bigr)}{\gamma_n\prod_{i=1}^{d}\bigl(\lambda_{m_i}(\gamma_n)-1\bigr)}
= \frac{\sum_{i=1}^{d}\bigl(a_n^{m_i-1}+m_i-1\bigr)}{a_n^{m-1}\prod_{i=1}^{d}q_{m_i}}.	
\end{equation}
By Lemma \ref{lem:partition}, we have
\[
\frac{\sum_{i=1}^{d}\lambda_{m_i}(\gamma_n)\bigl(1+B_{m_i}(\gamma_n)\bigr)}{\gamma_n\prod_{i=1}^{d}\bigl(\lambda_{m_i}(\gamma_n)-1\bigr)}\le\frac{m-1}{a_n^{m-1}},
\]
and then
\[
B_m(\gamma_n)
\le
\frac{m-1}{a_n^{m-1}}.
\]

Consider the composition
$(1,1,\dots,1)\in\mathcal{C}(m-1)$. 
Note that $B_1(\gamma_n)=0$, $\lambda_1(\gamma_n)=\gamma_n$ and $\lambda_1(\gamma_n)-1=a_n$. 
Then
\[
\frac{\sum_{i=1}^{d}\lambda_{1}(\gamma_n)\left(1+B_{1}(\gamma_n)\right)}{ \gamma_n\prod_{i=1}^{d}\left(\lambda_{1}(\gamma_n)-1\right)}=\frac{(m-1)\lambda_1(\gamma_n)
}{\gamma_n\bigl(\lambda_1(\gamma_n)-1\bigr)^{m-1}}=
\frac{m-1}{a_n^{m-1}}.
\]
Thus
\[
B_m(\gamma_n)
\ge
\frac{m-1}{a_n^{m-1}}.
\]

Consequently,
\[
B_m(\gamma_n)
=
\frac{m-1}{a_n^{m-1}}.
\]

It remains to determine the compositions attaining the maximum when
$m=n$. 
Let $(n_1,\ldots,n_d)\in\mathcal C(n-1)$
be a composition that attains the maximum of
$B_n(\gamma_n)$, that is, 
\[B_n(\gamma_n)=\frac{
\sum_{i=1}^{d}
\lambda_{n_i}(\gamma_n)
\bigl(1+B_{n_i}(\gamma_n)\bigr)
}{
\gamma_n
\prod_{i=1}^{d}
\bigl(\lambda_{n_i}(\gamma_n)-1\bigr)
}.\]
By Equation (\ref{eq:maxdecom}) and the fact $B_n(\gamma_n)=1$, we have
\[\sum_{i=1}^{d}\bigl(a_n^{n_i-1}+n_i-1\bigr)=(n-1)\prod_{i=1}^{d}q_{n_i}.\]
By the characterization of equality in 
Lemma~\ref{lem:partition},
we have $(n_1,\ldots,n_d)=(1,\dots,1)\in\mathcal{C}(n-1)$ or $(n_1,\ldots,n_d)=(n-1)\in\mathcal{C}(n-1)$.
Note that the exceptional case in Lemma~\ref{lem:partition} cannot occur here, since the
sum of the parts is $n-1$.

This completes the proof.
\end{proof}

\section{Proof of Theorem \ref{thm:main}}

The next lemma is the central comparison estimate.
\begin{lemma}
	\label{lem:comparison}
	Let $T$ be a tree of order $n$, let $u\in V(T)$, and let
	$z\in\mathbb{C}$ satisfy $|z|=r\ge 2$. Then
	\begin{equation*}
		\label{eq:comparison}
		|\Phi_{T-u}(z)|
		\le
		B_n(r)|\Phi_{T,u}(z)|.
	\end{equation*}
\end{lemma}

\begin{proof}
	We proceed by induction on $n$. If $n=1$, then $T-u$ is empty and the
	left-hand side is zero.
	
	Assume $n\ge 2$. Let $T_1,\dots,T_d$ be the components of $T-u$, rooted
	at the corresponding neighbors $u_1,\dots,u_d$ of $u$, and let
	$n_i=|V(T_i)|$. For convenience, set
	\[
	L_i=\Phi_{T_i,u_i}(z)\] and 
	\[
	F_i=\Phi_{T_i-u_i}(z).
	\]
	By the induction hypothesis,
	\begin{equation}\label{Ineq:between}
	|F_i|\le B_{n_i}(r)|L_i|.
	\end{equation}
	Since $\Phi_{T_i}(z)=L_i+F_i$ and
	$\Phi_{T-u}(z)=\sum_{i=1}^d\Phi_{T_i}(z)$, we obtain
	\[
	|\Phi_{T-u}(z)|\le\sum_{i=1}^d \bigl(|F_i|+|L_i|\bigr) 
	\le
	\sum_{i=1}^{d}\bigl(1+B_{n_i}(r)\bigr)|L_i|.
	\]
	On the other hand, the product formula \eqref{eq:rooted-product} and the
	reverse triangle inequality give
	\[
	|\Phi_{T,u}(z)|
	=
	r\prod_{i=1}^{d}|1+L_i|
	\ge
	r\prod_{i=1}^{d}\bigl(|L_i|-1\bigr).
	\]
	By Lemma~\ref{lowerbound}, we have 
	\[
	|L_i|
	\ge
	\lambda_{n_i}(r)
	>1,
	\]
	and then $|\Phi_{T,u}(z)|>0$.
	Consequently, 
\begin{equation}\label{Ineq:maxdecom}
	\frac{|\Phi_{T-u}(z)|}{|\Phi_{T,u}(z)|}
	\le
	\frac{
		\sum_{i=1}^d |L_i|\bigl(1+B_{n_i}(r)\bigr)
	}{
		r\prod_{i=1}^d(|L_i|-1)
	}
	\le
	\frac{
		\sum_{i=1}^d \lambda_{n_i}(r)\bigl(1+B_{n_i}(r)\bigr)
	}{
		r\prod_{i=1}^d(\lambda_{n_i}(r)-1)
	},
\end{equation}
	where the last inequality holds by Lemma~\ref{lem:H-coordinatewise}. 
	By the definition of
	$B_n(r)$, we have 
	\[
	\frac{
		\sum_{i=1}^d \lambda_{n_i}(r)\bigl(1+B_{n_i}(r)\bigr)
	}{
		r\prod_{i=1}^d(\lambda_{n_i}(r)-1)
	}\leq B_n(r).
	\]
	Thus 
	\[	\frac{|\Phi_{T-u}(z)|}{|\Phi_{T,u}(z)|}\le B_n(r).\]
	This completes the proof.	
\end{proof}

We now prove Theorem~\ref{thm:main}.

\begin{proof}[Proof of Theorem~\ref{thm:main}]

For $n=2$, the only tree is $K_{1,1}$, and
\[
\Phi_{K_{1,1}}(z)=z(z+2).
\]
Thus the unique nonzero root is $z=-2$, and
$|z|=2=1+(2-1)^{1/(2-1)}$.
Hence equality holds, as $n=2$ is even.

For $n=3$, the only tree is $K_{1,2}$, and
\[
\Phi_{K_{1,2}}(z)=z(z^2+2z+3).
\]
Its two nonzero roots have modulus
$\sqrt3<1+\sqrt2$,
so equality does not occur.

For $n=4$, the required radius is
$1+\sqrt[3]{3}$.
By Theorem~\ref{K1,3},  equality holds only for $K_{1,3}$.
Moreover,
\[
\Phi_{K_{1,3}}(z)
=
z\bigl((z+1)^3+3\bigr),
\]
and hence $z=-1-\sqrt[3]{3}$ is the unique root satisfying
\[
|z|=1+\sqrt[3]{3}.
\]

Thus the conclusion holds for $n=2,3,4$.
Now assume that $n\ge 5$. 

	Let $z\in\mathbb{C}$ satisfy $|z|=r>\gamma_n$, and choose any vertex
	$u\in V(T)$. Since $\gamma_n>2$,  Lemma~\ref{lem:comparison} gives
	\[
	|\Phi_{T-u}(z)|
	\le
	B_n(r)|\Phi_{T,u}(z)|.
	\]
	By Lemmas \ref{lemma:Bdecreasing} and \ref{lem:criticalB},
	\[
	B_n(r)<B_n(\gamma_n)= \frac{n-1}{a_n^{n-1}}=1.
	\]
	Moreover, by Lemma~\ref{lowerbound}, we have
\[
|\Phi_{T,u}(z)|
\ge
r(r-1)^{n-1}>0.
\]
    Consequently,
	\[
	|\Phi_{T-u}(z)|<|\Phi_{T,u}(z)|.
	\]
	By Equation \eqref{eq:decomposition} and the reverse triangle inequality, we have
	\[
	|\Phi_T(z)|
	\ge
	|\Phi_{T,u}(z)|-|\Phi_{T-u}(z)|
	>0.
	\]
	Thus $\Phi_T(z)\ne 0$ whenever $|z|>\gamma_n$.

We next characterize the boundary case. Suppose that $T$ is a tree of order $n$ such that $\Phi_T(z)=0$ and $|z|=\gamma_n$. 
For any vertex $u\in V(T)$, let
$T_1,\dots,T_d$ be the components of $T-u$. Let $u_i$ be the unique
neighbor of $u$ in $T_i$, and set
$n_i=|V(T_i)|$ and  $L_i=\Phi_{T_i,u_i}(z)$.
Then
\[
(n_1,\dots,n_d)\in\mathcal{C}(n-1).
\]
By Lemma~\ref{lowerbound}, we have $|L_i|
\ge
\lambda_{n_i}(\gamma_n)
>1$.
Hence, by Equation \eqref{eq:rooted-product}, we have 
\[
|\Phi_{T,u}(z)|
=
\gamma_n\prod_{i=1}^{d}|1+L_i|
\ge
\gamma_n\prod_{i=1}^{d}(|L_i|-1)
>0.
\]
Since $\Phi_T(z)=0$, Equation \eqref{eq:decomposition} yields
$\Phi_{T-u}(z)=-\Phi_{T,u}(z)$,
and therefore
\[
\frac{|\Phi_{T-u}(z)|}{|\Phi_{T,u}(z)|}=1.
\]
By Lemma~\ref{lem:comparison}, for each $i$,
\[
|\Phi_{T_i-u_i}(z)|
\le
B_{n_i}(\gamma_n)|\Phi_{T_i,u_i}(z)|.
\]
Replacing the induction-hypothesis Inequality \eqref{Ineq:between} in the proof of
Lemma~\ref{lem:comparison} by above inequalities and repeating the
derivation of Inequality \eqref{Ineq:maxdecom} with $r=\gamma_n$, we have
\[
1=
\frac{|\Phi_{T-u}(z)|}{|\Phi_{T,u}(z)|}\le
\frac{
\sum_{i=1}^{d}
\lambda_{n_i}(\gamma_n)
\bigl(1+B_{n_i}(\gamma_n)\bigr)
}{
\gamma_n
\prod_{i=1}^{d}
\bigl(\lambda_{n_i}(\gamma_n)-1\bigr)
}\le B_n(\gamma_n)=1.
\]
Hence
\[
\frac{
\sum_{i=1}^{d}
\lambda_{n_i}(\gamma_n)
\bigl(1+B_{n_i}(\gamma_n)\bigr)
}{
\gamma_n
\prod_{i=1}^{d}
\bigl(\lambda_{n_i}(\gamma_n)-1\bigr)
}
=
B_n(\gamma_n),
\]
so the composition
$(n_1,\ldots,n_d)\in\mathcal C(n-1)$
attains the maximum in the definition of $B_n(\gamma_n)$.
By the condition for 
equality in Lemma~\ref{lem:criticalB}, we have
$(n_1,\dots,n_d)=(1,\dots,1)$
or $(n_1,\dots,n_d)=(n-1)$,
which implies that $d_T(u)\in\{1,n-1\}$.
Since $u$ was arbitrary, every vertex of the tree $T$ has degree either $1$
or $n-1$. Thus $T\cong K_{1,n-1}$.

When $T\cong K_{1,n-1}$, we have
\[
\Phi_{K_{1,n-1}}(z)
=
z\bigl((z+1)^{n-1}+n-1\bigr).
\]
Hence, its nonzero subtree roots are
\[z_k=a_n\cos\left(\frac{2k-1}{n-1}\pi\right)-1+\mathrm{i}\cdot a_n\sin\left(\frac{2k-1}{n-1}\pi\right),\qquad
k=1,2,\ldots,n-1.\]
Then
\[
|z_k|^2=a_n^2+1-2a_n\cos\left(\frac{2k-1}{n-1}\pi\right).
\]
Thus $|z_k|=1+a_n$
if and only if
\[
\cos\left(\frac{2k-1}{n-1}\pi\right)=-1,
\]
which is equivalent to $2k=n$. 
Consequently, there exists a nonzero subtree root of
$K_{1,n-1}$ with modulus $1+a_n$ if and only if $n$ is even.
In this case $k=n/2$, and hence the unique boundary root is
$z_{n/2}=-1-a_n$.
Conversely, if $n$ is even, then $k=n/2$ is an integer and
$z_{n/2}=-1-a_n$ attains the upper bound.

This completes the proof.
\end{proof}

\section{A lower bound and sharpness results}

To establish a lower bound on the modulus of nonzero subtree roots, the following  result is essential.
\begin{lemma}\emph{\cite{1983On}}\label{L1}
	Let $T$ be a tree of order $n$. Then $s_1(T)=n$ and for $k=2,\ldots, n$,
	\[n-k+1\leq s_k(T)\leq \binom{n-1}{k-1}.\]
\end{lemma}
We now give the lower bound on the modulus of nonzero subtree roots. 
\begin{theorem}\label{T1}
	Let $T$ be a tree of order $n\geq 2$. Then every nonzero subtree root of $T$ satisfies
	\[ |z|>\sqrt[n-1]{n-1}-1.\]
\end{theorem}

\begin{proof}
The conclusion is immediate for $n=2$. We now assume that $n\ge3$.
	It suffices to prove that for  $0<|z|\leq\sqrt[n-1]{n-1}-1$,
	\[\left|\frac{\varPhi_{T}(z)}{z}\right|\geq 2.\]
	Note that $s_1(T)=n$ for any tree $T$ of order $n$. By the reverse triangle inequality and Lemma \ref{L1}, we have
	\begin{align*}
		\left|\frac{\varPhi_{T}(z)}{z}\right|&= \left|s_{n}(T)z^{n-1}+s_{n-1}(T)z^{n-2}+\cdots+s_{2}(T)z+n\right|\\
		&\geq n-\left(s_{n}(T)|z|^{n-1}+s_{n-1}(T)|z|^{n-2}+\cdots+s_{2}(T)|z|\right)\\
		&\geq n-\left(\binom{n-1}{n-1}|z|^{n-1}+\binom{n-1}{n-2}|z|^{n-2}+\cdots+\binom{n-1}{1}|z|+1-1\right)\\
		&=n+1-(|z|+1)^{n-1}\\
		&\geq 2.
	\end{align*}
	This completes the proof.
\end{proof}

We examine the sharpness of the upper and lower bounds on the modulus of subtree roots. In this section, 
for any integer $n\geq 2$, define
\[
a_n=(n-1)^{1/(n-1)}.
\]
Recall that for the star $K_{1,n-1}$,
\[
\Phi_{K_{1,n-1}}(z)
=
z\bigl((z+1)^{n-1}+n-1\bigr),
\]
and its nonzero subtree roots are
\[z_k=a_n\cos\left(\frac{2k-1}{n-1}\pi\right)-1+\mathrm{i}\cdot a_n\sin\left(\frac{2k-1}{n-1}\pi\right),\qquad
k=1,2,\ldots,n-1.\]

We first show that the upper bound in Theorem \ref{thm:main} is asymptotically sharp as
$n\to\infty$ through odd integers.
Suppose that $n$ is odd, and set
$k=(n-1)/2$.
Then
\[
\frac{(2k-1)\pi}{n-1}
=
\pi-\frac{\pi}{n-1},
\]
and hence
\[z_k=-a_n\cos\left(\frac{\pi}{n-1}\right)-1+\mathrm{i}\cdot a_n\sin\left(\frac{\pi}{n-1}\right).\]
It follows that
\begin{align*}
|z_k|^2=(1+a_n)^2-2a_n\left(1-\cos\left(\frac{\pi}{n-1}\right)\right).
\end{align*}
Therefore,
\[
\frac{|z_k|^2}{(1+a_n)^2}
=
1-
\frac{
	2a_n
	\left(
	1-\cos\left(\frac{\pi}{n-1}\right)
	\right)
}{
	(1+a_n)^2
}.
\]
Since
$a_n\rightarrow 1$ and 
$1-\cos\left(\frac{\pi}{n-1}\right)\rightarrow 0$ as $n\to\infty$, 
we have
\[
\lim_{\substack{n\to\infty\\ n\ {\rm odd}}}
\frac{|z_k|^2}{(1+a_n)^2}
=
1,
\]
and consequently 
\[
\lim_{\substack{n\to\infty\\ n\ {\rm odd}}}
\frac{|z_k|}{1+a_n}
=
1.
\]
Thus the upper bound in Theorem~1.3 is asymptotically sharp as
$n\to\infty$ through odd integers.

We next consider the lower bound in Theorem~4.2. Taking $k=1$, we
have
\[z_1=a_n\cos\left(\frac{\pi}{n-1}\right)-1+\mathrm{i}\cdot a_n\sin\left(\frac{\pi}{n-1}\right).\]
Then 
\begin{align*}
|z_1|^2=(a_n-1)^2+2a_n\left(1-\cos\left(\frac{\pi}{n-1}\right)\right),
\end{align*}
and
\[
\frac{|z_1|^2}{(a_n-1)^2}
=
1+
\frac{
	2a_n
	\left(
	1-\cos\left(\frac{\pi}{n-1}\right)
	\right)
}{
	(a_n-1)^2
}.
\]
Since $\frac{\ln(n-1)}{n-1}\rightarrow 0$ as $n\to\infty$,
we have 
\[
a_n-1=e^{\frac{\ln(n-1)}{n-1}}-1
\sim
\frac{\ln(n-1)}{n-1}.
\]
Moreover,
\[
1-\cos\left(\frac{\pi}{n-1}\right)
\sim
\frac{\pi^2}{2(n-1)^2}.
\]
Since $a_n\to1$ as $n\to\infty$, it follows that
\[
\lim_{n\to\infty}\frac{
2a_n\left(1-\cos\left(\frac{\pi}{n-1}\right)\right)
}{
(a_n-1)^2
}
=
\lim_{n\to\infty}\frac{
2\cdot \dfrac{\pi^2}{2(n-1)^2}
}{
\left(\dfrac{\ln(n-1)}{n-1}\right)^2
}
=
\lim_{n\to\infty}\left(\frac{\pi}{\ln(n-1)}\right)^2= 0.
\]
Therefore,
\[
\lim_{n\to\infty}
\frac{|z_1|^2}{(a_n-1)^2}
=
1,
\]
and consequently
\[
\lim_{n\to\infty}
\frac{|z_1|}{a_n-1}
=
1.
\]
Hence the lower bound in Theorem~4.2 is asymptotically sharp.

\section{Concluding remarks}

In this paper, we prove the Brown-Mol conjecture on subtree roots and completely characterize
the equality case. We also establish an asymptotically sharp lower
bound for the moduli of nonzero subtree roots.

Brown and Mol observed numerically in~\cite{Brown} that the subtree roots appear to be distributed
approximately around the point $-1/2$. They proposed the following conjecture, which
strengthens Theorem \ref{thm:main}, and verified it computationally for all
trees of order at most $18$.

\begin{conjecture}\emph{\cite{Brown}}\label{C2}
	If $T$ is a tree of order $n\geq 2$, then the subtree roots of $T$ are contained in the annulus
	\[\left \{ z\in \mathbb{C}: \frac{1}{2}\leq \left|z+\frac{1}{2}\right|\leq \frac{1}{2} + \sqrt[n-1]{n-1} \right\}.\]
\end{conjecture}

\begin{figure}[htbp]
	\centering
	\includegraphics[width=0.8\textwidth]{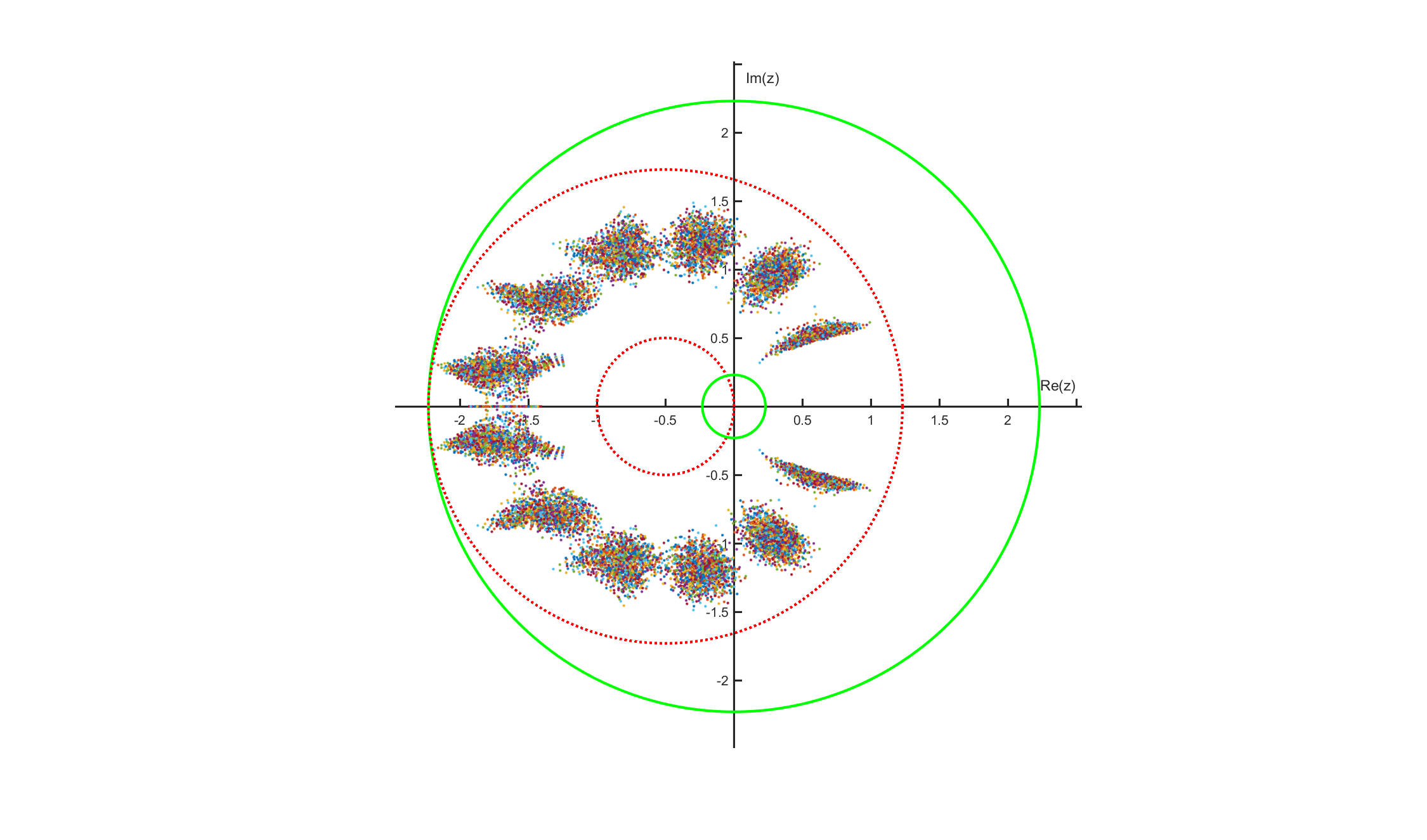}
	\caption{Subtree roots of all nonisomorphic trees of order $13$.
		Let $a_{13}=\sqrt[12]{12}\approx 1.2301$. The green circles
		$|z|=1+a_{13}$ and 
		$|z|=a_{13}-1$
		correspond to the bounds in
		Theorems~\ref{thm:main} and~\ref{T1}, respectively. The red circles $|z+1/2|=1/2$
		and 
		$|z+1/2|=1/2+a_{13}$
		correspond to the two bounds in Conjecture \ref{C2}, respectively.}
	\label{fig:roots-order-13}
\end{figure}

For comparison, Figure~\ref{fig:roots-order-13} shows the subtree roots of
all nonisomorphic trees of order $13$, together with the boundary circles
arising from Theorems~\ref{thm:main} and~\ref{T1} and those appearing in
Conjecture~\ref{C2}. The figure is included to illustrate the 
relationship between the bounds proved in this paper and the stronger
translated annulus conjectured by Brown and Mol. We emphasize that
Brown and Mol \cite{Brown} carried out more extensive computations for all trees of
order at most $18$.

The method developed in this paper does not appear to extend
directly to Conjecture \ref{C2}. 
A different approach may therefore be required in the future.

\section*{Acknowledgements}

Conjecture \ref{C1} has been of sustained interest to the second author
over the past several years. The ideas leading to this work were
developed through discussions during the 13th Workshop on Matrix and
Graph Spectra Theory held in Urumqi, Xinjiang, China, and the work was
completed during subsequent discussions at Xinjiang Normal University.

This work is supported by the National Natural Science Foundation of China (Nos. 12571379, 12571360, 12571366), Scientific Research Start-Up Foundation of Jimei University (No. ZQ2024116), Fujian Provincial Department of Education (No. JAT251074) and Natural Science
Foundation of Xiamen Municipality (No.~3502Z202673038).

\section*{Declarations}
\noindent
{\bf Conflict of interest}
The authors declare that they have no conflict of interest.


\end{document}